\documentclass[10pt,a4paper,leqno,twoside]{amsart}
\usepackage[a4paper,left=3.1cm,top=4.3cm,bottom=4.6cm,right=3.1cm]{geometry}
\usepackage[T1]{fontenc}
\usepackage{amssymb}
\usepackage{amsmath,amsthm,dsfont}
\usepackage{amsfonts}
\usepackage{latexsym}
\usepackage[cp1250]{inputenc}
\usepackage{graphicx}
\usepackage{indentfirst}
\usepackage{enumerate}
\usepackage{geometry}
\usepackage{enumitem}
\usepackage{fancyhdr}
\usepackage[usenames,dvipsnames]{xcolor}

\usepackage{hyperref}
\hypersetup{colorlinks,linkcolor={blue},citecolor={red}}

\makeatletter

\renewenvironment{proof}[1][\proofname]{\par \pushQED{\qed} \normalfont
  \topsep6\p@\@plus6\p@ \trivlist \itemindent\z@
  \item[\hskip\labelsep\bfseries
    #1\@addpunct{.}]\ignorespaces
}{
  \popQED\endtrivlist\@endpefalse
}

    \@addtoreset{equation}{section} %resetuje licznik po rozpoczęciu nowej sekcji
    \renewcommand{\theequation}{{\thesection}.\@arabic\c@equation} %dodaje kropkę

\renewcommand{\subsection}[1]{\stepcounter{subsection}\vspace{2mm}
\par{\bf \thesubsection.  \bfseries #1.}}

\def\section{\@ifstar\unnumberedsection\numberedsection}
\def\numberedsection{\@ifnextchar[%]
  \numberedsectionwithtwoarguments\numberedsectionwithoneargument}
\def\unnumberedsection{\@ifnextchar[%]
  \unnumberedsectionwithtwoarguments\unnumberedsectionwithoneargument}
\def\numberedsectionwithoneargument#1{\numberedsectionwithtwoarguments[#1]{#1}}
\def\unnumberedsectionwithoneargument#1{\unnumberedsectionwithtwoarguments[#1]{#1}}
\def\numberedsectionwithtwoarguments[#1]#2{%
  \ifhmode\par\fi
  \removelastskip
  \vskip 4ex\goodbreak
  \refstepcounter{section}%
  \noindent
  \begingroup
  \leavevmode\centering\scshape\bfseries
  \thesection.
  #2
  \par
  \endgroup
  \vskip 1ex\nobreak
  \addcontentsline{toc}{section}{%
    \protect\numberline{\thesection}%
    #1}%
  }
\def\unnumberedsectionwithtwoarguments[#1]#2{%
  \ifhmode\par\fi
  \removelastskip
  \vskip 2ex\goodbreak
  \noindent
  \begingroup
  \leavevmode\centering\scshape\bfseries
  \leavevmode\centering\scshape\bfseries
  #2
  \par
  \endgroup
  \vskip 1ex\nobreak
  \addcontentsline{toc}{section}{%
    #1}%
}
\makeatother

\theoremstyle{plain}
\newtheorem{Th}{Theorem}[section]
\newtheorem{Prop}[Th]{Proposition}
\newtheorem{Lem}[Th]{Lemma}

\theoremstyle{definition}
\newtheorem{Def}[Th]{Definition}
\newtheorem{Rem}[Th]{Remark}

\def\noo{\partial}

\def\bf{\textbf}
\def\it{\textit}
\def\te{\textnormal}

\def\leq{\leqslant}
\def\geq{\geqslant}
\def\R{{\mathds R}}
\def\N{{\mathds N}}

\def\C{{\mathcal{C}}}
\def\L{{\mathcal{L}}}
\def\B{\textit{I\!B}}
\def\D{{\mathrm{dom}}\,}
\def\E{{\mathrm{epi}}\,}

\begin{document}
\vspace*{-4mm}
\title{{\bf\sc On nonuniqueness of solutions of Hamilton-Jacobi-Bellman equations}}

\vspace*{-4mm}

\author{\bf{Arkadiusz Misztela \textdagger}}

\vspace*{-4mm}

\thanks{\textdagger\, Institute of Mathematics, University of Szczecin, Wielkopolska 15, 70-451 Szczecin, Poland; e-mail:\\ arkadiusz.misztela@usz.edu.pl\vspace{0cm}}

\begin{abstract}
\noindent An example of a nonunique solution of the Cauchy problem of Hamilton-Jacobi-Bellman (HJB) equation  with surprisingly regular Hamiltonian is presented. The Hamiltonian $H(t,x,p)$ is locally Lipschitz continuous with respect to all variables, convex in $p$ and with linear growth with respect to $p$ and  $x$. The HJB equation possesses two distinct lower semicontinuous solutions with the same final conditions; moreover, one of them is the value function of the corresponding Bolza problem.  The definition of lower semicontinuous solution was proposed by Barron-Jensen~\cite{B-J} and Frankowska~\cite{HF}. Using the example an analysis and comparison of assumptions in some uniqueness results in HJB equations is provided.

\vspace{4mm}\hspace{-1cm}
\hspace*{-0.55cm} viscosity solution.

\vspace{2mm}\hspace{-1cm}

\end{abstract}

\maketitle

\pagestyle{myheadings}  \markboth{\small{ARKADIUSZ MISZTELA}
}{\small{ON NONUNIQUENESS OF SOLUTIONS}}

\vspace{-7mm}

%%%%%%%%%%%%%%%%%%%%%%%%%%%%%%%%%%%%%%%%%%%%%%%%%%%%%%%%%%%%%%%%%%%%%%%%%%%%%%
%%%%%%%%%%%%%%%%%%%%%%%%%%%%%%%%%%%%%%%%%%%%%%%%%%%%%%%%%%%%%%%%%%%%%%%%%%%%%%
%%%%%%%%%%%%%%%%%%%%%%%%%%%%%%%%%%%%%%%%%%%%%%%%%%%%%%%%%%%%%%%%%%%%%%%%%%%%%%
%%%%%%%%%%%%%%%%%%%%%%%%%%%%%%%%%%%%%%%%%%%%%%%%%%%%%%%%%%%%%%%%%%%%%%%%%%%%%%

%który

\section{Introduction}

\noindent The classical Cauchy problem for the  Hamilton-Jacobi-Bellman equation is a partial differential equation with the final condition
\begin{equation}\label{row}
\begin{array}{rll}
-U_{t}+ H(t,x,-U_{x})=0 &\te{in} & ]0,T[\,\times\,\R^n, \\[0mm]
U(T,x)=g(x) & \te{in} & \R^n.
\end{array}
\tag*{HJB}
\end{equation}

\noindent If the Hamiltonian $H$ is convex in the gradient variable, then there are relations between solutions of \ref{row} and calculus of variation problems involving a function dual to $H$. This function, called the Lagrangian and denoted by $L$,  is the Legendre-Fenchel transform of  $H$ in its gradient variable:
 \begin{equation}\label{tr1}
 L(t,x,v)=\sup_{p\in\R^{n}}\,\{\,\langle v,p\rangle-H(t,x,p)\,\}.
\end{equation}
The value function of a  calculus of variation  problem $V:[0,T]\times\R^n\to\R\cup\{+\infty\}$ is defined by 
\begin{equation}\label{FW-def}
V(t_0,x_0)= \inf_{\begin{array}{c}
\scriptstyle x(\cdot)\,\in\,\mathcal{A}\left([t_0,T],\R^n\right)\\[-1mm]
\scriptstyle x(t_0)=x_0
\end{array}}\,\Big\{\,g(x(T))+\int_{t_0}^TL(t,x(t),\dot{x}(t))\,dt\,\Big\},
\end{equation}
where $\mathcal{A}\!\left([t_0,T],\R^n\right)$ denotes the space of all absolutely continuous functions from $[t_0,T]$ into $\R^n$.
 If the value function is  differentiable, it is well-known that it satisfies \ref{row} in the classical sense. However, in many situations the value function is not differentiable. Then the solution of the HJB equation must be defined in nonsmooth sense in such a way that under quite general assumptions on $H$ and $g$, $V$ is the unique solution of HJB. Since we use the nonsmooth analysis we need a notion of a subgradient.  For a vector $v\in\R^n$ and a function $f:\R^n\rightarrow\R\cup\{+\infty\}$,  $v$ is a subgradient of $f$ at $x\in\D f$, written $v\in\noo f(x)$, if
%\vspace{1mm}
\begin{equation}\label{subgra}
\liminf\limits_{y\rightarrow x}\frac{f(y)-f(x)-\langle v,y-x\rangle}{|y-x|}\geqslant 0.
\end{equation}

In 1990 Baron-Jensen \cite{B-J} and Frankowska \cite{HF} introduced extended viscosity solutions to semicontinuous functions for Hamiltonian that is convex in the gradient variable and provided a uniqueness result. Frankowska \cite{HF}  called these solutions \it{lower semicontinuous solutions}.

\begin{Def}\label{lsc-solutions}  A function $U:[0,T]\times\R^n\rightarrow\R\cup\{+\infty\}$ is lower semicontinuous solution of the \ref{row} equation if it satisfies the following:
\vspace{-0.8mm}
\begin{enumerate}[leftmargin=7.5mm]
\item[\te{\bf{(i)}}] $U$ is lower semicontinuous and $U(T,x)=g(x)$ for all $x\in\R^n$;
\item[\te{\bf{(ii)}}] for any $(t,x)\in \D U$, for all $(p_{t},p_{x})\in\noo U(t,x)$, one has 
\begin{equation}\label{lsc-sol}
\left\{\begin{array}{lll}
-p_{t}+H(t,x,-p_{x})\geq 0 & \te{if} & 0\leq t<T, \\[1mm]
-p_{t}+H(t,x,-p_{x})\leq 0 & \te{if} & 0<t\leq T.
\end{array}\right.
\end{equation}
\end{enumerate}
\end{Def}

The main goal of the  paper is to present an example of two distinct lower semicontinuous solutions of the  \ref{row} equation with surprisingly regular Hamiltonian.  Understanding the role that the Lipschitz-type condition plays in theorems about uniqueness
of solution of \ref{row} is also important. The proposed Hamiltonian $H\!:[0,T]\!\times\!\R\times\!\R\rightarrow\R$ satisfies the \it{classical assumptions} i.e. firstly, it is convex with respect to~$p$, secondly, it increases linearly in $p$ and $x$, i.e. $|H(t,x,p)|\leq 2|p|$ for any $t\in[0,T]$, $x,p\in\R$, thirdly, it is locally Lipschitz continuous, i.e.
\begin{equation}\label{ass-llc}
\begin{array}{l}
\forall\, r>0\;\; \exists\, k>0\;\;\forall\, t,s\in[0,T]\;\; \forall\, x,y\in r\B\;\; \forall\, p,q\in r\B  \\[0mm]
|H(t,x,p)-H(s,y,q)|\leq k(|t-s|+|x-y|+|p-q|),
\end{array}
\tag{LLC}
\end{equation}
where  $\B$ is the closed unit ball. In addition, we  show that one of the indicated lower semicontinuous solutions is the value function given by~(\ref{FW-def}). In general, for the uniqueness of \ref{row} solutions, one needs some stronger Lipschitz-type condition, that we shall study further in connection to the results of uniqueness.

Frankowska~\cite{HF} proved that, the value function is the unique lower semicontinuous solution of the  \ref{row} equation if the Hamiltonian meets the  classical assumptions and it is positively homogeneous in~$p$, i.e.~$\forall _{r>0}\;H(t,x,rp)=rH(t,x,p)$. Actually, the result of Frankowska does not require local Lipschitz continuity with respect to the triple~$(t,x,p)$. It is enough to assume it is satisfied with respect to state variable~$x$ only. The example of nonuniqueness of solution of~\ref{row} introduced in the current paper, does not contradict the result of Frankowska as the Hamiltonian in our example, fulfills the classical assumptions, but it is not positively homogeneous in~$p$.

Earlier Ishii \cite[Thm. 2.5]{HI} and  Crandall-Lions \cite[Thm. VI.1]{C-L-87} had proved the uniqueness of viscosity solutions of \ref{row} in the class of continuous functions. They had assumed instead of linear growth in~$p$ and $x$ of Hamiltonian, the following condition
\begin{equation}\label{ass-000}
\begin{array}{l}
\exists\, C>0\;\;\forall\, t\in[0,T]\;\; \forall\, x\in\R^n\;\;\forall\, p,q\in\R^n  \\[0mm]
|H(t,x,p)-H(t,x,q)|\;\leq\; C(1+|x|)|p-q|.
\end{array}
%\exists\, C>0\;\forall\, t\in[0,T]\; \forall\, x\in\R^n\;\forall\, p,q\in\R^n\;\;
%|H(t,x,p)\!-\!H(t,x,q)|\leq C(1+|x|)|p-q|
\end{equation}
One can show that if the Hamiltonian is convex in~$p$ and possesses a linear growth of the form $H(t,x,p)\leq 2(1+|x|)|p|$, then condition (\ref{ass-000}) holds with constant~$C=2$. Therefore the Hamiltonian from our  example of nonuniqueness also satisfies~(\ref{ass-000}). Next, the results  from papers\linebreak \cite[Thm. 2.5]{HI} and  \cite[Thm. VI.1]{C-L-87} require the Lipschitz-type condition for the Hamiltonian:
\begin{equation}\label{ass-hc}
\begin{array}{l}
\forall\, r>0\;\; \exists\, k>0\;\;\forall\, t,s\in[0,T]\;\; \forall\, x,y\in r\B\;\; \forall\, p\in\R^n  \\[0mm]
|H(t,x,p)-H(s,y,p)|\leq k(1+|p|)(|t-s|+|x-y|),
\end{array}
\tag{SLC}
\end{equation}
that is derived from the optimal control problem. The meaning of~(\ref{ass-hc}) in the optimal control problems is discussed in the following papers \cite{F-S,AM3,AM4,FR}. The results in \cite[Thm. 2.5]{HI} and  \cite[Thm. VI.1]{C-L-87} do not require the convexity of the Hamiltonian in~$p$, but the uniqueness of  \ref{row} solutions is obtained in the class of continuous functions. Using these results, Bardi and Capuzzo-Dolcetta \cite[Chap. V, Thm. 5.16]{B-CD} showed the uniqueness of  \ref{row} solutions in the class of lower semicontinuous functions assuming additionally the convexity of the Hamiltonian in~$p$. Because the Hamiltonian in our example of nonuniqueness satisfies the  classical assumptions, also the condition (\ref{ass-000}) is satisfied. It means that in the uniqueness results, the key point is  (\ref{ass-hc}), at least in the case of lower semicontinuous solutions of the  \ref{row} equation.

In order to understand better the reason for nonuniqueness of the solution of the   \ref{row} equation  given in our example, we need to recall the Loewen-Rockafellar condition from \cite{L-R-94} that is more general version of~(\ref{ass-hc}).  This condition was used
 by Loewen-Rockafellar  \cite{L-R,L-R-97} to study necessary conditions satisfied by optimal solutions of Bolza problem.
Galbraith \cite{G} proved that, the value function is the unique lower semicontinuous solution of~\ref{row} assuming the Loewen-Rockafellar condition \cite{L-R-94}. This uniqueness result has, in general, the same nature as the uniqueness result of Dal Maso-Frankowska \cite{DM-F-V}. The Hamiltonian from our example of nonuniqueness is constructed using the function $\varphi$. 
In Section~\ref{section-rh} we show that if $\varphi$ is replaced by a function that is sufficiently regular, then we obtain the Hamiltonian that does not satisfy the condition~(\ref{ass-hc}), while it satisfies the Loewen-Rockafellar condition~\cite{L-R-94}. Therefore, by virtue of classical results, we are not able to say if after the mentioned change, we get the uniqueness of the solution of~\ref{row} or not. However, using the Galbraith result~\cite{G}, we know that \ref{row} with this Hamiltonian has the unique solution.  Thus, the \ref{row} equation can have the unique solution even if 
the Hamiltonian does not satisfy the condition~(\ref{ass-hc}). Moreover, the Hamiltonian from our example of nonuniqueness shows the difference between
(\ref{ass-llc}), (\ref{ass-hc}) and Loewen-Rockafellar condition~\cite{L-R-94}. 
We know that the condition (\ref{ass-llc}) does not guarantee uniqueness of  the solution to the \ref{row} equation, so the natural question can be stated -- 
what  extra conditions the Hamiltonian should satisfy in order the  \ref{row} equation has
the unique solution? The answer to this question is obtained when analysing conditions
(\ref{ass-llc}), (\ref{ass-hc}) and Loewen-Rockafellar condition~\cite{L-R-94}. 
 Namely, some extra conditions in the interdependence between space and subgradient for large values of the subgradient are mandatory (see Sect. \ref{Exi_uni_thm}).

In the literature, an example of nonuniqueness of the solution of the equation of~\ref{row} is known. Crandall and Lions in their fundamental article~\cite{C-L} give an example of nonuniqueness of viscosity solution of the transport equation 
\vspace{-1mm}
\begin{equation*}
-U_{t}+b(x)\cdot (-U_x)=0,\quad U(T,x)=g(x).
\end{equation*}  

\vspace{-1mm}
\noindent In this example, the function $b(\cdot)$ is bounded and continuous, but is  not locally Lipschitz continuous. Therefore the Hamiltonian~$H(t,x,p)=b(x)\,p$ is convex, continuous and satisfies the condition (\ref{ass-000}), but it does not satisfy (\ref{ass-llc}). The solutions in the Crandall-Lions example are continuous, and in our example, they are lower semicontinuous. Now we consider an easy example of nonuniqueness of lower semicontinuous solutions of the transport equation. 
Let $b(x)=x^2$ and $g\equiv 0$. Then functions $V\equiv 0$ and  $U(t,x)=0$, $(T-t)x\leq 1$, $U(t,x)=1$, $(T-t)x>1$ are lower semicontinuous solutions of the transport equation. 
In this example, the Hamiltonian $H(t,x,p)=x^2p$ satisfies (\ref{ass-llc}), (\ref{ass-hc}) and Loewen-Rockafellar condition~\cite{L-R-94}, but it does not satisfy (\ref{ass-000}) or equivalently it does not have a linear growth in $x$. 
It means that these conditions play important roles in the uniqueness
of lower semicontinuous solutions of \ref{row}.  
Notice that  Hamiltonians from the above two examples of nonuniqueness do not satisfy the
classical assumptions, but the Hamiltonian from our example of nonuniqueness does.
In the transport equation we cannot find the example of nonuniqueness as ours. Because on one hand the  \ref{row} equation with the  Hamiltonian satisfying the
classical assumptions and positively homogeneous in~$p$ has the unique solution 
(see Frankowska~\cite{HF}). On the other hand conditions (\ref{ass-llc}), (\ref{ass-hc}) and Loewen-Rockafellar condition~\cite{L-R-94} are equivalent, 
if the Hamiltonian is  convex in $p$, positively homogeneous in~$p$ and continuous. 
Therefore, we cannot find differences between them while considering Hamiltonians
from the transport equation.

Summarizing, the above examples of nonuniqueness show that the key role in the uniqueness of viscosity solutions of~\ref{row} is played by the Lipschitz-type conditions originated in optimization instead of local Lipschitz continuity originated in  differential equations theory. Therefore our result is proper for optimization problems.

We have found the example of nonuniqueness presented here while generalizing a
result of Plaskacz-Quincampoix \cite{P-Q}. This generalization was published in \cite{AM2}.

%\pagebreak

%%%%%%%%%%%%%%%%%%%%%%%%%%%%%%%%%%%%%%%%%%%%%%%%%%%%%%%%%%%%%%%%%%%%%%%%%%%%%%
%%%%%%%%%%%%%%%%%%%%%%%%%%%%%%%%%%%%%%%%%%%%%%%%%%%%%%%%%%%%%%%%%%%%%%%%%%%%%%
%%%%%%%%%%%%%%%%%%%%%%%%%%%%%%%%%%%%%%%%%%%%%%%%%%%%%%%%%%%%%%%%%%%%%%%%%%%%%%
%%%%%%%%%%%%%%%%%%%%%%%%%%%%%%%%%%%%%%%%%%%%%%%%%%%%%%%%%%%%%%%%%%%%%%%%%%%%%%

\section{Existence and uniqueness theorem}\label{Exi_uni_thm}
\noindent In this section we present the well-known theorem on existence and uniqueness of lower 
semicontinuous solutions of the \ref{row} equation. We discuss results of the paper on the basis of this theorem. First, we introduce basic assumptions on Hamiltonian:

\begin{enumerate}[leftmargin=9mm]
\item[\te{(H1)}] $H:[0,T]\times\R^{n}\times\R^{n}\rightarrow\R$  is continuous;
\item[\te{(H2)}] $H(t,x,p)$ is convex with respect to $p$ for every $(t,x)\in[0,T]\times\R^n$;
\item[\te{(H3)}] there exists a constant $C>0$ such that for all
$(t,x,p)\in[0,T]\times\R^n\times\R^n$  the following\\ inequality is satisfied
$H(t,x,p)\leqslant C(1+|x|)|p|$.
\end{enumerate}

\noindent Condition (H3) is called  \it{a linear growth in $p$ and $x$ of the Hamiltonian}. 
Basing on (H1)--(H3) we can state the following theorem on existence and uniqueness:

\begin{Th}\label{lfm-Th100}
We suppose that $H$ satisfies \te{(H1)--(H3)} and  \te{(\ref{ass-hc})}. Let $g$
be lsc and bounded from below, and $V$ be the value function associated with  $g$ and $L$, where $L$ is given by \te{(\ref{tr1})}. Then the value function $V$ is bounded from below lower semicontinuous  solution of \te{\ref{row}}. Moreover,
if $U$ is  bounded from below lower semicontinuous  solution of \te{\ref{row}}, then $U=V$ on $[0,T]\times\R^n$.
\end{Th}

\noindent We shall show that Theorem  \ref{lfm-Th100} is a particular case of Theorem 2.2 from   \cite{G}. To this end, we need to discuss the subgradient characterization of the condition  (\ref{ass-hc}) and the Loewen-Rockafellar condition from \cite{L-R-94}. 
A subgradient is defined for the function given on the entire Euclidean space (see Def. \ref{subgra}), so to use a subgradient of the Lagrangian $L$ we extend it in the following way: $L(t,x,v):=L(0,x,v)$ for $t<0$ and  $L(t,x,v):=L(T,x,v)$ for $t>T$.
\begin{equation}\label{ass-11}
\begin{array}{l}
\te{For every}\; r>0\; \te{there exists}\; k>0\; \te{such that at every point}\; (t,x,v)\\ 
\in\R\times r\B\times \R^n,\;  \te{every}\; (w_1,w_2,p)\in\noo L(t,x,v)\; \te{the inequality}\\
|(w_1,w_2)|\leq k(1+|p|)\; \te{holds.}
\end{array}
\end{equation}
Condition (\ref{ass-11}) is equivalent to  (\ref{ass-hc}), if Hamiltonian satisfies  (H1)--(H3). Moreover, one can prove that the condition  (\ref{ass-11}) 
is a subgradient characterization of Lipschitz continuity of the multifunction
$(t,x)\rightarrow\E L(t,x,\cdot)$ in the Hausdorff's sense. Besides, it is easy to see that the condition (\ref{ass-11}) implies the following one: 
\begin{equation}\label{ass-1}
\begin{array}{l}
\te{For every}\; r>0\; \te{there exists}\; k>0\; \te{such that at every point}\; (t,x,v)\\ 
\in\R\times r\B\times \R^n,\;  \te{every}\; (w_1,w_2,p)\in\noo L(t,x,v)\; \te{we have}\\
|(w_1,w_2)|\leq k(1+|v|+|L(t,x,v)|)(1+|p|).
\end{array}
\end{equation}
The condition (\ref{ass-1}) is a subgradient characterization 
of the Aubin type continuity of multifunction $(t,x)\to\E L(t,x,\cdot)$. This kind of Aubin continuity was introduced by Loewen and Rockafellar in  \cite[Def. 2.3, (b)]{L-R-94}. However,
a subgradient characterization can be found in\linebreak  \cite[Prop. 3.4]{G1}. 
Galbraith \cite{G} obtained the uniqueness result  assuming convexity of Hamiltonian\linebreak with respect to $p$, a mild growth  of Hamiltonian with respect to $p$ (see (A1) from \cite{G}) and slightly more general Lipschitz-type condition than (\ref{ass-1}) (see (A2) from \cite{G}).
Thus, if we replace the condition (\ref{ass-hc}) by (\ref{ass-1}) in Theorem \ref{lfm-Th100}, then the claim of Theorem  \ref{lfm-Th100} still holds. 
In  \cite{G,G1,L-R-94,L-R,L-R-97} the authors use a limiting subgradient that is larger than the regular subgradient we use. We know that replacing the regular  subgradient 
by the limiting subgradient one the condition  (\ref{ass-1}) is unchanged if we assume that 
Hamiltonian is continuous.  We also know that the condition (\ref{ass-1}) implies  (\ref{ass-llc}) (see \cite{G}, Prop. 2.4). Therefore, assuming conditions (H1)--(H3) we obtain 
\begin{equation}\label{imp-101}
\te{(\ref{ass-hc})}\;\; \Longrightarrow\;\;\te{(\ref{ass-1})}\;\; \Longrightarrow\;\; \te{(\ref{ass-llc})}.
\end{equation}
Further on, we will show that the implications  (\ref{imp-101}) cannot be reversed. 
Moreover, we prove that the Lipschitz continuity  (\ref{ass-llc}) raising from 
differential equations is not sufficient for the uniqueness of lower semicontinuous solution
of \ref{row}. However, conditions (\ref{ass-hc}) and (\ref{ass-1}) 
of optimization problems are sufficient for uniqueness.

%%%%%%%%%%%%%%%%%%%%%%%%%%%%%%%%%%%%%%%%%%%%%%%%%%%%%%%%%%%%%%%%%%%%%%%%%%%%%%
%%%%%%%%%%%%%%%%%%%%%%%%%%%%%%%%%%%%%%%%%%%%%%%%%%%%%%%%%%%%%%%%%%%%%%%%%%%%%%
%%%%%%%%%%%%%%%%%%%%%%%%%%%%%%%%%%%%%%%%%%%%%%%%%%%%%%%%%%%%%%%%%%%%%%%%%%%%%%
%%%%%%%%%%%%%%%%%%%%%%%%%%%%%%%%%%%%%%%%%%%%%%%%%%%%%%%%%%%%%%%%%%%%%%%%%%%%%%

\section{Example of  Hamiltonian}\label{section-ph}

\noindent In this section we define a Hamiltonian  and discuss its regularity. 
In the next section we show that the \ref{row} equality with this Hamiltonian
does not have the unique lower semicontinuous solution.
We define an auxiliary function
$\varphi:[0,T]\times\R\rightarrow\R$  by the formula 
\begin{equation}\label{pf}
\varphi(t,x)=\sqrt{|t-x|}\,\exp\left(2\sqrt{|t-x|}\right).
\end{equation}

The Hamiltonian $H:[0,T]\times\R\times\R\rightarrow\R$ is defined by the formula
\begin{equation}\label{ex-H}
H(t,x,p)=\left\{
\begin{array}{ccl}
\:0 & \;\it{if}\;\; & 2|p|\leq\frac{1}{\varphi(t,x)},\;t \not=x, \\[1mm]
\:2|p|-\frac{1}{\varphi(t,x)} & \;\it{if}\;\; & 2|p|>\frac{1}{\varphi(t,x)},\;t \not=x, \\[1mm]
\:0 & \;\it{if}\;\; & p\in\R,\; t=x.
\end{array}
\right.
\end{equation}

\noindent It is not difficult to see that the Hamiltonian $H$ given by (\ref{ex-H})
is continuous on $[0,T]\times\R\times\R$, convex with respect to  $p$ for each  $(t,x)\in[0,T]\times\R$ and has a linear growth in $p$ and $x$, i.e. $|H(t,x,p)|\leq 2|p|$ for all $t\in[0,T]$, $x,p\in\R$, so it
satisfies  (H1)--(H3).

\begin{Th}\label{stw-hc}
The Hamiltonian $H$ given by \te{(\ref{ex-H})} satisfies locally the Lipschitz continuity, i.e. for each  $(t_0,x_0,p_0)\in[0,T]\times\R\times\R$ there exist numbers  $r,k>0$ such that
\begin{equation}\label{lwl}
\begin{array}{l}
\forall\, t,s\in\B(t_0,r)\;\; \forall\, x,y\in\B(x_0,r)\;\; \forall\, p,q\in\B(p_0,r)\\
|H(t,x,p)-H(s,y,q)|\leq k(|t-s|+|x-y|+|p-q|).
\end{array}
\end{equation}
\end{Th}

\begin{proof} For $\theta\geq 0$ we define an auxiliary function  $f$ by the formula $f(\theta)=\sqrt{\theta}\exp(2\sqrt{\theta})$. We notice that the function  $f$ is
increasing and $f(|t-x|)=\varphi(t,x)$ for each $t\in[0,T]$, $x\in\R$. Furthermore, the function  $f$ on $[a, b]$
satisfies Lipschitz continuity if $0<a<b$. We fix $t_0\in[0,T]$, $x_0\in\R$ and $p_0\in\R$, and consider two cases.

\bf{Case 1.} Let $p_0\in\R$ and $t_0=x_0$. We define $r:=1/[2\exp(4)(1+2|p_0|)^2]$ and $k:=1$. We notice that  $|p|\leq r+|p_0|$ for each $p\in\B(p_0,r)$ and $|t-x|\leq 2r$ for each $t\in\B(t_0,r)$, $x\in\B(x_0,r)$, with $r\leq 1/2$. Therefore, for $t\not=x$
we obtain\vspace{1mm}
$$2|p|\leq 2(r+|p_0|)\leq 1+2|p_0|=\frac{1}{\sqrt{2r}\exp(2)}\leq\frac{1}{f(2r)}\leq \frac{1}{f(|t-x|)}
=\frac{1}{\varphi(t,x)}.$$\vspace{0.1mm}

\noindent By the definition of the Hamiltonian  $H(t,x,p)=0$ if  $p\in\R$, $t=x$ and $2|p|\leq 1/\varphi(t,x)$, $t\not=x$. Therefore,  $H(t,x,p)=0$ for each $t\in\B(t_0,r)$, $x\in\B(x_0,r)$, $p\in\B(p_0,r)$, so we have  (\ref{lwl}).

\bf{Case 2.} Let $p_0\in\R$ and $t_0\not=x_0$. We define $r$ by the formula $r:=|t_0-x_0|/3$, then $r\leq|t-x|\leq 5r$ for each  $t\in\B(t_0,r)$, $x\in\B(x_0,r)$. Let  $l$ be the Lipschitz
constant of the function $f$ on $[r,5r]$. We define $k$ by $k:=2+l/f^2(r)$.

Since $t\not= x$ for each  $t\in\B(t_0,r)$, $x\in\B(x_0,r)$, then by the definition of the Hamiltonian (\ref{ex-H}) we obtain the following relations:

\begin{enumerate}%[leftmargin=7mm]
%\vspace{1mm}
\item[\bf{(a)}] For $2|p|\leq 1/\varphi(t,x)$ and $2|q|\leq 1/\varphi(s,y)$ we have $\te{LS}(\ref{lwl})= 0 \leq \te{RS}(\ref{lwl})$.
\item[\bf{(b)}] For $2|p|\geq 1/\varphi(t,x)$ and $2|q|\geq 1/\varphi(s,y)$
we have inequalities
\begin{eqnarray*}
\te{LS}(\ref{lwl})
&\leq& 2|p-q|+\frac{1}{f(|t-x|)f(|s-y|)}|f(|t-x|)-f(|s-y|)|\\
&\leq& 2|p-q|+\frac{l}{f^2(r)}\left(|t-s|+|x-y|\right)\;\;\leq\;\; \te{RS}(\ref{lwl}).
\end{eqnarray*}
\item[\bf{(c)}] For  $2|p|\geq 1/\varphi(t,x)$ and $2|q|\leq 1/\varphi(s,y)$
we have inequalities
\begin{equation*}
\te{LS}(\ref{lwl}) \;\leq\; 2|p|-\frac{1}{\varphi(t,x)}+\frac{1}{\varphi(s,y)}-2|q|\;\leq\; \te{RS}(\ref{lwl}).
\end{equation*}
\end{enumerate}
The consequence of cases \bf{(a)-(c)} is the inequality (\ref{lwl}).
\end{proof}

The Lagrangian $\!L\!:\![0,T]\!\times\!\R\!\times\!\R\rightarrow\R\cup\{\!+\!\infty\}\!$ given by the formula \eqref{tr1} has the following form:
\begin{equation}\label{ex-L}
L(t,x,v)=\left\{
\begin{array}{ccl}
+\infty & \;\it{if}\;\; & |v|> 2,\;t\not=x, \\[1.5mm]
\frac{|v|}{2\varphi(t,x)} & \;\it{if}\;\; & |v|\leq 2,\;t\not=x, \\[2mm]
0 & \;\it{if}\;\; & v=0,\; t=x,\\[0.5mm]
+\infty & \;\it{if}\;\; & v\not=0,\;t=x.
\end{array}
\right.
\end{equation}
Now we prove that Lagrangian \te{(\ref{ex-L})} does not satisfy the condition (\ref{ass-1}).
The proof of this fact shows how condition (\ref{ass-1}) is violated for large values of gradients. It implies that the second implication in (\ref{imp-101}) cannot be reversed.
 In Section \ref{section-rh} we will see that large values of gradients not necessarily violate condition (\ref{ass-1}), if $\varphi$ is sufficiently regular.

\begin{Prop}
Lagrangian \te{(\ref{ex-L})} does not satisfy the condition  \te{(\ref{ass-1})}.
\end{Prop}

\begin{proof} Let the Lagrangian $L$
be given by the formula (\ref{ex-L}).
 Then for $t>x$, $v\in\,]0,2[$ and $(w_1,w_2,p)\in\noo L(t,x,v)$ we have
\begin{equation*}
p=\frac{1}{2\varphi(t,x)}, \qquad
-w_1=w_2=\frac{v}{2\varphi(t,x)}\left[\frac{1}{2(t-x)}+\frac{1}{\sqrt{(t-x)}}\right].
\end{equation*}
Therefore, the left and right hand sides of the inequality (\ref{ass-1}) are given by
\begin{equation*}
\te{LS}(\ref{ass-1})=\sqrt{2}\,w_2, \qquad \te{RS}(\ref{ass-1})=k\left(1+v+\frac{v}{2\varphi(t,x)}\right)
\left(1+\frac{1}{2\varphi(t,x)}\right).
\end{equation*}
Let $t_n-x_n\rightarrow 0+$ and $v_n=2\varphi(t_n,x_n)$. If the inequality (\ref{ass-1})
is satisfied, then for large $n\in\N$ 
\begin{eqnarray*}
\sqrt{2}\left[\frac{1}{2(t_n-x_n)}+\frac{1}{\sqrt{(t_n-x_n)}}\right] &\leq& 2k(1+\varphi(t_n,x_n))
\left(1+\frac{1}{2\varphi(t_n,x_n)}\right).
\end{eqnarray*}
Multiplying the above inequality by $2(t_n-x_n)$ we have the inequality
\begin{eqnarray*}
\sqrt{2}+2\sqrt{2(t_n-x_n)} &\leq&  2k\left[1+\varphi(t_n,x_n)\right]\left(2(t_n-x_n)+
\frac{\sqrt{t_n-x_n}}{\exp\left(2\sqrt{t_n-x_n}\,\right)}\right).
\end{eqnarray*}
Passing to the limit, we obtain a contradiction.
\end{proof}

%%%%%%%%%%%%%%%%%%%%%%%%%%%%%%%%%%%%%%%%%%%%%%%%%%%%%%%%%%%%%%%%%%%%%%%%%%%%%%
%%%%%%%%%%%%%%%%%%%%%%%%%%%%%%%%%%%%%%%%%%%%%%%%%%%%%%%%%%%%%%%%%%%%%%%%%%%%%%
%%%%%%%%%%%%%%%%%%%%%%%%%%%%%%%%%%%%%%%%%%%%%%%%%%%%%%%%%%%%%%%%%%%%%%%%%%%%%%
%%%%%%%%%%%%%%%%%%%%%%%%%%%%%%%%%%%%%%%%%%%%%%%%%%%%%%%%%%%%%%%%%%%%%%%%%%%%%%

\section{An example of nonuniqueness}
\noindent In this section we present two different, bounded, lower semicontinuous
solutions of the  \ref{row} equation with the Hamiltonian $H$ given by (\ref{ex-H}) and
the terminal condition  $g:\R\rightarrow\R$ given by 
\begin{equation}\label{wpo}
g(x)=\left\{
\begin{array}{ccl}
\:\exp\left(-2\sqrt{x-T}\:\right)-1 & \;\it{if}\;\; & x\geq T, \\[1mm]
\:1 & \;\it{if}\;\; &  x<T.
\end{array}
\right.
\end{equation}
Now we introduce the lemma that is needed in the proof of the inequality
 (\ref{lsc-sol}).
For the function  $f:\R^n\rightarrow\R\cup\{+\infty\}$ and $z\in\D f$ we define
\begin{equation}\label{pwk}
\Delta f(z)(e)=\lim_{\;\;\tau\rightarrow 0+}\frac{f(z+\tau e)-f(z)}{\tau}.
\end{equation}

\begin{Lem}[\te{\cite[Prop. 6.4.8]{A-F}}]
Let the function $F:\R^2\rightarrow\R\cup\{+\infty\}$ and a point $(t,x)\in\D F$ be given. Assume that   $\Delta F(t,x)(\upsilon,\omega)$ exists. Then for each $(p_t,p_x)\in\noo F(t,x)$ we have
\begin{equation}\label{wpon}
\Delta F(t,x)(\upsilon,\omega)\;\geq\;p_t\upsilon+p_x\omega.
\end{equation}
\end{Lem}

\noindent In this paper we define some notions on the whole Euclidean space, that while applied to   $U(t,x)$ given on  $[0,T]\times\R^n$ we extend  $U(t,x)$ by setting $+\infty$ for $(t,x)\notin[0,T]\times\R^n$. In addition to this, we introduce the following notation
$\psi(t,x):=\frac{1}{\varphi(t,x)}$ for $t\not=x$.

\subsection{First solution}
Let the function $U:[0,T]\times\R\rightarrow\R$ be given by the formula
\begin{equation}\label{sol-1}
U(t,x)=\left\{
\begin{array}{ccl}
\:\exp\left(-2\sqrt{x-t}\:\right)-1 & \;\it{if}\; & x\geq t, \\[1mm]
\:1 & \;\it{if}\; &  x<t.
\end{array}
\right.
\end{equation}

\begin{Th}\label{stw-roz}
The function  $U$ given by \te{(\ref{sol-1})} is bounded  lower semicontinuous solution of the  \te{\ref{row}} equation with the Hamiltonian \te{(\ref{ex-H})} and the terminal condition \te{(\ref{wpo})}.
\end{Th}

%\vspace{-7mm}
\begin{proof}
It is not difficult to notice that the function $U$ is lower semicontinuous,
bounded  and $U(T,x)=g(x)$. We will prove that the function $U$ satisfies conditions (\ref{lsc-sol}). To this end, we consider five cases.

\bf{Case 1.} Let $x>t$ and $t\in [0,T[$. Then the equality $\Delta U(t,x)(1,0)=\psi(t,x)$ holds. Therefore, from the inequality (\ref{wpon}) we obtain $\psi(t,x)\geq p_t$ for each  $(p_t,p_x)\in\noo U(t,x)$. Moreover, $\Delta U(t,x)(0,1)=-\psi(t,x)$ and $\Delta U(t,x)(0,-1)=\psi(t,x)$. Therefore, from the inequality (\ref{wpon}) we have
 $-\psi(t,x)\geq p_x$ and $\psi(t,x)\geq -p_x$ for each $(p_t,p_x)\in\noo U(t,x)$.   Since
$-p_t\geq -\psi(t,x)$ and $-p_x =\psi(t,x)$ for each $(p_t,p_x)\in\noo U(t,x)$,
then by the definition of $H$ we have \vspace{0.2cm}
\begin{eqnarray*}
-p_t+H(t,x,-p_x) &=& -p_t -2p_x- \psi(t,x)\\
&\geq& -\psi(t,x) +2\psi(t,x)- \psi(t,x)\;\;=\;\;0.
\end{eqnarray*}

\vspace{2mm}

\bf{Case 2.} Let $x>t$ and $t\in\,]0,T]$. Then the equality $\Delta U(t,x)(-1,0)=-\psi(t,x)$
holds. Therefore, from the inequality (\ref{wpon}) we have $-\psi(t,x)\geq -p_t$
for each $(p_t,p_x)\in\noo U(t,x)$. Moreover, $\Delta U(t,x)(0,1)=-\psi(t,x)$ and $\Delta U(t,x)(0,-1)=\psi(t,x)$. Therefore, from the inequality (\ref{wpon}) we get $-\psi(t,x)\geq p_x$ and $\psi(t,x)\geq -p_x$ for each $(p_t,p_x)\in\noo U(t,x)$.  Since $-p_t\leq -\psi(t,x)$ and $-p_x =\psi(t,x)$ for each $(p_t,p_x)\in\noo U(t,x)$, then by the definition of $H$ we have \vspace{0.2cm}
\begin{eqnarray*}
-p_t+H(t,x,-p_x) &=& -p_t -2p_x- \psi(t,x)\\
&\leq& -\psi(t,x) +2\psi(t,x)- \psi(t,x)\;\;=\;\;0.
\end{eqnarray*}

\bf{Case 3.} Let $x=t$ and $t\in[0,T]$. Then the equality $\Delta U(t,x)(0,1)=-\infty$
holds.  If $(p_t,p_x)\in\noo U(t,x)$, then from the inequality (\ref{wpon}) we have the
contradiction $-\infty=\Delta U(t,x)(0,1)\geq p_x\in\R$. Therefore $\noo U(t,x)=\emptyset$.
%that finishes the proof.

\vspace{3mm}

\bf{Case 4.} Let $x<t$ and $t\in[0,T[$. Then the equality $\Delta U(t,x)(1,0)=0$ holds. Therefore, from the inequality (\ref{wpon}) we obtain $0\geq p_t$ for each  $(p_t,p_x)\in\noo U(t,x)$. Moreover, $\Delta U(t,x)(0,1)=0$ and $\Delta U(t,x)(0,-1)=0$. Therefore, from the inequality (\ref{wpon}) we have
 $0\geq p_x$ and $0\geq -p_x$ for each $(p_t,p_x)\in\noo U(t,x)$.   Since
$-p_t\geq 0$ and $-p_x =0$ for each $(p_t,p_x)\in\noo U(t,x)$,
then by the definition of $H$ we have
\begin{equation*}
-p_t+H(t,x,-p_x) \;\geq\;  0 +H(t,x,0) \;=\; 0.
\end{equation*}

\vspace{3mm}

\bf{Case 5.} Let $x<t$ and $t\in\,]0,T]$. Then the equality $\Delta U(t,x)(-1,0)=0$
holds. Therefore, from the inequality (\ref{wpon}) we have $0\geq -p_t$
for each $(p_t,p_x)\in\noo U(t,x)$. Moreover, $\Delta U(t,x)(0,1)=0$ and $\Delta U(t,x)(0,-1)=0$. Therefore, from the inequality (\ref{wpon}) we get $0\geq p_x$ and $0\geq -p_x$ for each $(p_t,p_x)\in\noo U(t,x)$.  Since $-p_t\leq 0$ and $-p_x =0$ for each $(p_t,p_x)\in\noo U(t,x)$,
then by the definition of $H$ we have 
\begin{equation*}
-p_t+H(t,x,-p_x) \;\leq\;   0 +H(t,x,0)\;=\; 0,
\end{equation*}
that finishes the proof.
\end{proof}

%\vspace*{3mm}
\subsection{Second solution}
Let the function $V:[0,T]\times\R\rightarrow\R$ be given by the formula
\begin{equation}\label{sol-2}
V(t,x)=\left\{
\begin{array}{ccl}
\:U(t,x) & \;\it{if}\;\; & x\geq t, \\[1mm]
\:1-\exp\left(-2\sqrt{t-x}\:\right) & \;\it{if}\;\; & 2t-T\leq x< t, \\[1mm]
\:1 & \;\it{if}\;\; &  x<2t-T.
\end{array}
\right.
\end{equation}

%\vspace{1mm}
\begin{Th}\label{stw-roz2}
The function $V$ given by \te{(\ref{sol-2})} is bounded  lower semicontinuous solution
of the  \te{\ref{row}} equation with the Hamiltonian \te{(\ref{ex-H})} and the terminal condition \te{(\ref{wpo})}.
\end{Th}

\begin{proof}
It is not difficult to notice that the function  $V$ is lower semicontinuous, bounded  and  $V(T,x)=g(x)$. We prove that the function  $V$ satisfies conditions  (\ref{lsc-sol}). Since $V(t,x)=U(t,x)$ for  $x\geq t$ and $x<2t-T$, than by the Theorem \ref{stw-roz} it is sufficient to show that  $V$ satisfies conditions (\ref{lsc-sol}), when $2t-T\leq x< t$. To do it we consider two cases.

%\vspace{0.2cm}
\bf{Case 1.} Let $2t-T\leq x<t$ and $t\in\;]0,T[$. Then $\Delta V(t,x)(-1,-2)=\psi(t,x)$ and $\Delta V(t,x)(1,2)=-\psi(t,x)$. Therefore, from the inequality (\ref{wpon}) we have
 $\psi(t,x)\geq-p_t-2p_x$ and $-\psi(t,x)\geq p_t+2p_x$ for each $(p_t,p_x)\in\noo V(t,x)$. Moreover, $\Delta V(t,x)(0,1)=-\psi(t,x)$, so by the inequality (\ref{wpon})
we get $-\psi(t,x)\geq p_x$ for each $(p_t,p_x)\in\noo V(t,x)$. Since $-p_x\geq \psi(t,x)$ and  $\psi(t,x)= -p_t-2p_x$ for all $(p_t,p_x)\in\noo V(t,x)$, then from the definition
of $H$  we have
\vspace{0.2cm}
\begin{eqnarray*}
-p_t+H(t,x,-p_x) &=& -p_t -2p_x- \psi(t,x)\\
&=& \psi(t,x)- \psi(t,x)\;\;=\;\;0.
\end{eqnarray*}

\vspace{0.2cm}
\bf{Case 2.} Let $2t-T\leq x<t$ and $t=0$. Then $\Delta V(t,x)(1,2)=-\psi(t,x)$. Therefore, from the inequality (\ref{wpon}) we have $-\psi(t,x)\geq p_t+2p_x$ for all $(p_t,p_x)\in\noo U(t,x)$. Besides,  $\Delta V(t,x)(0,1)=-\psi(t,x)$, so from the inequality (\ref{wpon}) we have $-\psi(t,x)\geq p_x$ for all $(p_t,p_x)\in\noo V(t,x)$. Since $-p_x\geq \psi(t,x)$
and  $-p_t-2p_x\geq \psi(t,x)$ for all $(p_t,p_x)\in\noo V(t,x)$, then by the definition
of $H$, %we have
\vspace{0.2cm}
\begin{eqnarray*}
-p_t+H(t,x,-p_x) &=& -p_t -2p_x- \psi(t,x)\\
&\geq& \psi(t,x)- \psi(t,x)\;\;=\;\;0,
\end{eqnarray*}
%\vspace{0.1cm}
\noindent that ends the proof.
\end{proof}

%%%%%%%%%%%%%%%%%%%%%%%%%%%%%%%%%%%%%%%%%%%%%%%%%%%%%%%%%%%%%%%%%%%%%%%%%%%%%%
%%%%%%%%%%%%%%%%%%%%%%%%%%%%%%%%%%%%%%%%%%%%%%%%%%%%%%%%%%%%%%%%%%%%%%%%%%%%%%
%%%%%%%%%%%%%%%%%%%%%%%%%%%%%%%%%%%%%%%%%%%%%%%%%%%%%%%%%%%%%%%%%%%%%%%%%%%%%%
%%%%%%%%%%%%%%%%%%%%%%%%%%%%%%%%%%%%%%%%%%%%%%%%%%%%%%%%%%%%%%%%%%%%%%%%%%%%%%

\section{The value function}\label{section-VF}

\noindent In this section we show that the function $V$ given by (\ref{sol-2})
is the value function, in the sense of definition \eqref{FW-def},
corresponding to the Lagrangian (\ref{ex-L}) and the terminal condition (\ref{wpo}). 
To this purpose, we use the methods of \cite{AM1}, whose scheme is as follows:
We construct a sequence of Hamiltonians $(H_n)_{n\in\N}$
such that $H_n$ satisfy (H1)--(H3) together with (\ref{ass-hc}) and $H_n\searrow H$. Then, by  \cite[Cor. 3.5]{AM1},
the value functions $V_n$ corresponding to \ref{row} with Hamiltonians $H_n$
and terminal conditions  $g_n=g$ converge to the value function $V$ ($V_n\nearrow V$).
Since Hamiltonians $H_n$ additionally, satisfy the condition   (\ref{ass-hc}), then by Theorem \ref{lfm-Th100}, the
value functions $V_n$ are unique lower semicontinuous
solutions. Therefore, to find the value function  $V$, one needs to find solutions $U_n$ of
equations \ref{row} with Hamiltonians $H_n$ and the terminal condition  $g$ and then take the limit $U_n=V_n\to V$.

\subsection{Approximation of  the Hamiltonian}
Let the function $\sigma:[0,+\infty[\rightarrow\R$ be given by $\sigma(z)=\sqrt{z}$, and
functions $\sigma_n:[0,+\infty[\;\rightarrow\,\R$ by $\sigma_n(z)=\sqrt{z}$ if $ z\geq 1/n$ and $\sigma_n(z)=1/\sqrt{n}$ if $ 0\leq z<1/n$.
We notice that functions $\sigma_n$ satisfy  locally the Lipschitz continuity and $\sigma_n\searrow\sigma$.
Using functions~$\sigma_n$ we define functions  $\varphi_n:[0,T]\times\R\rightarrow\R$ by
%\vspace{-3mm}
\begin{equation*}\label{apro-2}
\varphi_n(t,x)=\sigma_n(|t-x|)\exp[2\sigma_n(|t-x|)].
\end{equation*}
Then functions $\varphi_n$ also satisfy locally the Lipschitz continuity and $\varphi_n\searrow\varphi$.

Hamiltonians $H_n:[0,T]\times\R\times\R\rightarrow\R$ are defined by the formula
\begin{equation}\label{ex-H_n}
H_n(t,x,p)=\left\{
\begin{array}{ccl}
\:0 & \;\it{if}\;\; & 2|p|\leq\frac{1}{\varphi_n(t,x)}, \\[1mm]
\:2|p|-\frac{1}{\varphi_n(t,x)} & \;\it{if}\;\; & 2|p|>\frac{1}{\varphi_n(t,x)}.
\end{array}
\right.
\end{equation}
It is not difficult to notice that Hamiltonians $H_n$ given by (\ref{ex-H_n})
are continuous, convex with respect to $p$  and have linear growth in $p$ and $x$,
so they satisfy conditions (H1)--(H3). Moreover $H_n\searrow H$, because $\varphi_n\searrow\varphi$. Similarly, for the proof of Case 2 in Theorem~\ref{stw-hc} we can prove that Hamiltonians~(\ref{ex-H_n}) satisfy the condition (\ref{ass-hc}).

\subsection{Solutions $\pmb{U_n}$ of \ref{row} with  $\pmb{H_n}$} Let   $U_{n}:[0,T]\times\R\rightarrow\R$ be given by formula:

\begin{enumerate}%[leftmargin=7mm]
\vspace{2mm}
\item[]\hspace{6mm}$\pmb{U_n(t,x)=}$
\vspace{-4mm}
\item[] \hspace{4cm} For $x\geq 2t-T$ we have
\vspace{2mm}
\item[\bf{(a)}] if $|t-x|\leq 1/n$ and $T-2t+x\geq 1/n$, then
\begin{eqnarray*}
\;\;\;\;U_n(t,x)&=&(t-x)\sqrt{n}\exp(-2/\sqrt{n}\,)+(1+1/\sqrt{n}\,)\exp(-2/\sqrt{n}\,)-1,
\end{eqnarray*}
\item[\bf{(b)}] if $|t-x|\leq 1/n$ and $T-2t+x\leq 1/n$, then
\begin{eqnarray*}
\!\!U_n(t,x)&=&\exp(-2\sqrt{T-2t+x}\,)+(T-t)\sqrt{n}\exp(-2/\sqrt{n}\,)-1,
\end{eqnarray*}
\item[\bf{(c)}] if $t-x\geq 1/n$ and $T-2t+x\geq 1/n$, then
\begin{eqnarray*}
\hspace{-6mm}U_n(t,x)&=&2(1+1/\sqrt{n}\,)\exp(-2/\sqrt{n}\,)-\exp(-2\sqrt{t-x}\,)-1,
\end{eqnarray*}
\item[\bf{(d)}] if $t-x\geq 1/n$ and $T-2t+x\leq 1/n$, then
\begin{eqnarray*}
\hspace{-13mm}U_n(t,x)&=&[1+\sqrt{n}(T-2t+x)+1/\sqrt{n}\,]\exp(-2/\sqrt{n}\,)\\
\hspace{-13mm}&+&\exp(-2\sqrt{T-2t+x}\,)-\exp(-2\sqrt{t-x}\,)-1,
\end{eqnarray*}
\item[\bf{(e)}] if $1/n\leq x-t$, then $U_n(t,x)=\exp(-2\sqrt{x-t}\,)-1$.
\vspace{1mm}
\item[] \hspace{4cm} For $x< 2t-T$ we have $U_n(t,x)=1$.
\vspace{0mm}
\end{enumerate}

\begin{Prop}\label{lsc-sol-nnnn}
Let Hamiltonians $H_n$ be given by \te{(\ref{ex-H_n})}, the terminal condition $g$ by \te{(\ref{wpo})},
and the function $V$ by  \te{(\ref{sol-2})}. Then functions $U_n$ given by the above
formula are bounded  lower semicontinuous solutions of equations
 \te{\ref{row}} with Hamiltonians $H_n$ and the terminal condition $g$, moreover $U_n\rightarrow V$.
\end{Prop}

\begin{proof}
It is not difficult to notice that functions  $U_n$ are bounded (\,i.e. $-1\leq U_n(\,\cdot\,,\,\cdot\,)\leq 1$\,)
and $U_n(T,x)=g(x)$. We can prove that functions $U_n$ are  lower semicontinuous
on $[0,T]\times\R$. Furthermore, functions $U_n$ are differentiable on  $A=\{(t,x)\in\;]0,T[\times\R: x> 2t-T\}$ and $B=\{(t,x)\in\;]0,T[\times\R: x< 2t-T\}$, moreover $\noo U_n(t,x)=\emptyset$ for $x=2t-T$ and $t\in[0,T]$, in addition  $U_n$ satisfy conditions (\ref{lsc-sol}) with Hamiltonians $H_n$ and $U_n\rightarrow V$.
\end{proof}

%%%%%%%%%%%%%%%%%%%%%%%%%%%%%%%%%%%%%%%%%%%%%%%%%%%%%%%%%%%%%%%%%%%%%%%%%%%%%%
%%%%%%%%%%%%%%%%%%%%%%%%%%%%%%%%%%%%%%%%%%%%%%%%%%%%%%%%%%%%%%%%%%%%%%%%%%%%%%
%%%%%%%%%%%%%%%%%%%%%%%%%%%%%%%%%%%%%%%%%%%%%%%%%%%%%%%%%%%%%%%%%%%%%%%%%%%%%%
%%%%%%%%%%%%%%%%%%%%%%%%%%%%%%%%%%%%%%%%%%%%%%%%%%%%%%%%%%%%%%%%%%%%%%%%%%%%%%

\section{Regularity of  Hamiltonian}\label{section-rh}

\noindent Let $\C$ be a family of continuous functions $\varphi:[0,T]\times\R\to[0,+\infty[$ that satisfy the condition $\varphi(t,x)=0\,\Leftrightarrow\, t=x$. Then it is easy to prove that the  Hamiltonian $H$ given by \te{(\ref{ex-H})} with $\varphi\in\C$ is well-defined and satisfies (H1)--(H3). Next, by $\L$ we denote a subfamily of $\C$ that contains locally Lipschitz functions. Notice that $\varphi$ given by the formula (\ref{pf}) belongs to $\C$, but does not belong to $\L$. 
The example of a function, that is contained in $\L,$ is $\varphi(t,x)=|t-x|$.
 
 In this section we prove that the Lagrangian $L$ given by \te{(\ref{ex-L})} with $\varphi\in\L$ satisfies the Loewen-Rockafellar condition (\ref{ass-1}). However, it can be shown easily that its Hamiltonian does not satisfy the Lipschitz-type condition (\ref{ass-hc}). It means that we cannot reverse the first implication in (\ref{imp-101}). Moreover, by the result of Galbraith \cite{G} it follows that \ref{row} equation with $H$ given by \te{(\ref{ex-H})} with $\varphi\in\L$ has the unique solution. 

So, it could be said that the nonuniqueness of \ref{row} is due to a particular choice of $\varphi\in\C$ and that choosing a different finction $\varphi\in\L\subset\C$ one can get the unique solution of \ref{row}.

Let the function $\varphi$ be given by (\ref{pf}) and $w(\,\cdot\,,r)$ be
 modulus of continuity of the $\varphi$ on the set $[0,T]\times r\B$. Then the
following proposition holds.

\begin{Prop}\label{wlpq1}
Let the Lagrangian $L$ be given by  \te{(\ref{ex-L})} with the function $\varphi$ belongs to $\C$. Moreover, let $w(\,\cdot\,,r)$ be modulus of continuity of  $\varphi$ . Then for  every $t,s\in[0,T]$ and $x,y\in r\B$, every $v\in \D L(t,x,\cdot)$ there exists $\nu\in \D L(s,y,\cdot)$  such that
\begin{itemize}
\item[\te{\bf{(i)}}] $|\nu-v|\leqslant 2(1+|v|+|L(t,x,v)|)\,w(|s-t|+|y-x|,r)$;
\item[\te{\bf{(ii)}}] $L(s,y,\nu)\leqslant L(t,x,v)+2(1+|v|+|L(t,x,v)|)\,w(|s-t|+|y-x|,r)$.
\end{itemize}
\end{Prop}

\begin{proof}
To prove the proposition we consider 3 cases.

\bf{Case 1.} Let $t\not=x$ and $\varphi(s,y)/\varphi(t,x)\leq 1$. Then for $v\in \D L(t,x,\cdot)$ we put $\nu=v\,\varphi(s,y)/\varphi(t,x)$. We notice that $\nu\in \D L(s,y,\cdot)$. Moreover $\te{LS(ii)}\leq L(t,x,v)\leq \te{RS(ii)}$ and $\te{LS(i)}=2L(t,x,v)|\varphi(s,y)-\varphi(t,x)|\leq \te{RS(i)}$.

\bf{Case 2.} Let $t\not=x$ and $\varphi(s,y)/\varphi(t,x)> 1$. Then for $v\in \D L(t,x,\cdot)$ we put $\nu=v$. We notice that $\nu\in \D L(s,y,\cdot)$. Moreover $\te{LS(i)}=0\leq\te{RS(i)}$ and $\te{LS(ii)}\leq L(t,x,v)\leq\te{RS(ii)}$.

\bf{Case 3.} Let $t=x$. If $v\in \D L(t,x,\cdot)$, then $v=0$. Put $\nu=0$. We notice
that $\nu\in \D L(s,y,\cdot)$.  Moreover $\te{LS(i)}=0\leq\te{RS(i)}$ and 
$\te{LS(ii)}=0\leq \te{RS(ii)}$.

Therefore, the proposition is proven.
\end{proof}

\begin{Prop}\label{swl-gss2}
The condition \te{(\ref{ass-1})} holds, if the following condition  is true:
\begin{itemize}[leftmargin=8mm]
\item[\te{\bf{(A)}}] For every $r>0$ there exists $k>0$ such that for every $t,s\in[0,T]$ and $x,y\in r\B$, every $v\in \D L(t,x,\cdot)$ there exists $\nu\in \D L(s,y,\cdot)$  such that
\begin{itemize}[leftmargin=7mm]
\item[\te{\bf{(i)}}] $|\nu-v|\leqslant k(1+|v|+|L(t,x,v)|)(|s-t|+|y-x|)$;
\item[\te{\bf{(ii)}}] $L(s,y,\nu)\leqslant L(t,x,v)+k(1+|v|+|L(t,x,v)|)(|s-t|+|y-x|)$.
\end{itemize}
\end{itemize}
\end{Prop}

\begin{proof} We extend $L$ in the following way: $L(t,x,v):=L(0,x,v)$ for  $t<0$ and  $L(t,x,v):=L(T,x,v)$ for $t>T$. Fix $r>0$ and choose $k>0$ for $1+r$ in such a way that
the condition (A) holds. Let $t\in[0,T]$, $x\in r\B$, $v\in\R^n$ and $(w_1,w_2,p)\in\noo L(t,x,v)$. Without loss of generality we can assume that $(w_1,w_2)\not=0$. Let $(t_n,x_n):=(t,x)+(w_1,w_2)/[n|(w_1,w_2)|]$, then $x_n\in (1+r)\B\,$. Since $v\in\D L(t,x,\cdot)$, then there exist $v_n\in\D L(t_n,x_n,\cdot)$ such that
\begin{itemize}[leftmargin=6.5mm]
\item[(i)] $|v_n-v|\leqslant 2k(1+|v|+|L(t,x,v)|)\,|(t_n,x_n)-(t,x)|$,
\item[(ii)] $L(t_n,x_n,v_n)\leqslant L(t,x,v)+2k(1+|v|+|L(t,x,v)|)\,|(t_n,x_n)-(t,x)|$.
\end{itemize}
We put $b_n:=n(v_n-v)$ and notice that  (i) implies that $|b_n|\leq 2k(1+|v|+|L(t,x,v)|)$. Therefore, a sequence $\{b_n\}_{n\in\N}$ is bounded, so there exists a subsequence
(denoted again by) $b_n\rightarrow b$. Obviously, the following inequality is satisfied
\begin{eqnarray}\label{nwa1}
|b|\leq 2k(1+|v|+|L(t,x,v)|).
\end{eqnarray}
Since $(w_1,w_2,p)\in\noo L(t,x,v)$, then  the property \cite[s. 301]{R-W} or \cite[Chap. 6]{A-F} and (ii)
imply
\begin{eqnarray}\label{oa1}
\nonumber\left\langle\,(w_1,w_2,p),\left(\frac{(w_1,w_2)}{|(w_1,w_2)|},b\right)\,\right\rangle &\leq& d\, L(t,x,v)\left(\frac{(w_1,w_2)}{|(w_1,w_2)|},b\right)\\
\nonumber &\leq& \liminf_{n}\frac{L(t_n,x_n,v_n)-L(t,x,v)}{|(t_n,x_n)-(t,x)|}\\
&\leq& 2k(1+|v|+|L(t,x,v)|).
\end{eqnarray}

\noindent From the inequality \eqref{oa1} and \eqref{nwa1} for $(w_1,w_2,p)\in\noo L(t,x,v)$
we obtain
\begin{eqnarray*}
|(w_1,w_2)| &\leq& 2k(1+|v|+|L(t,x,v)|)+|p||b|\\
&\leq& 2k(1+|v|+|L(t,x,v)|)(1+|p|).
\end{eqnarray*}
So the proposition is proven.
\end{proof}

\begin{Rem}
The only difference between the assertion of Proposition \ref{wlpq1} and the condition \te{(A)}
is contained in  modulus.
From Proposition \ref{wlpq1} and \ref{swl-gss2} we obtain that the Lagrangian $L$ given by \te{(\ref{ex-L})} with $\varphi\in\L$  satisfies the Loewen-Rockafellar condition  \te{(\ref{ass-1})}.
\end{Rem}

\end{document}